\documentclass{amsart}
\usepackage{amsmath, amsthm, amssymb,esint}
\usepackage{fullpage}
\usepackage{xcolor}
\definecolor{green1}{rgb}{0.0, 0.5, 0.0}
\usepackage{hyperref}
\usepackage{soul}
\usepackage{enumitem}
\usepackage{tikz}
\usetikzlibrary{positioning}

\usepackage{pgfplots}

\pgfplotsset{}
\usetikzlibrary{patterns}

\newtheorem{theorem}{Theorem}

\newtheorem{lemma}{Lemma}
\newtheorem{proposition}[lemma]{Proposition}

\newtheorem{definition}[lemma]{Definition}
\newtheorem{remark}[lemma]{Remark}

\numberwithin{lemma}{section}

\newcommand{\cA}{{\mathcal A}}

\newcommand{\jD}{\langle D \rangle}

\numberwithin{equation}{section}

\newcommand{\R}{{\mathbb R}}

\renewcommand{\R}{\mathbb R}

\newcommand{\tu}{\tilde{u}}

\begin{document}

\title{Local well-posedness for quasilinear problems: a primer}

\author{Mihaela Ifrim}
\address{Department of Mathematics, University of Wisconsin, Madison}
\email{ifrim@wisc.edu}

\author{ Daniel Tataru}
\address{Department of Mathematics, University of California at Berkeley}
\email{tataru@math.berkeley.edu}

\begin{abstract}
Proving local well-posedness for quasilinear problems
in pde's presents a number of difficulties, some of which are 
universal and others of which are more problem specific.
On one hand, a common standard for what well-posedness should mean  has existed for a long time, going back to Hadamard.
On the other hand, in terms of getting there, there are by now both many variations, but also many misconceptions.

The aim of these expository notes is to collect a number of both classical
and more recent ideas in this direction, and to assemble them into a cohesive road-map that can be then adapted to the reader's problem of choice. 

\end{abstract}

\maketitle

\setcounter{tocdepth}{1}
\tableofcontents

\maketitle

\section{Introduction}

Local well-posedness is the first question to ask for any evolution problem in partial differential equations. These notes, prepared by the authors for a summer school at MSRI~\cite{msri} in 2020, aim to
discuss ideas and strategies for local well-posedness in quasilinear and fully nonlinear evolution equations, primarily of hyperbolic type. We hope to persuade the reader that
the structure presented here should be adopted as the standard for proving these results.
Of course, there are many possible variations, and we try to point out some of them in our many remarks. While a few of the ideas here can be found in several of the classical books, see e.g. \cite{Taylor},\cite{Hormander}, \cite{Chemin},\cite{Sogge}, some of the others have appeared only in articles devoted to specific problems, 
and have never been collected together, to the best of our knowledge.

\subsection{Nonlinear evolutions} 
For our exposition we will adopt a two track structure, where we will  
broadly discuss ideas for a general problem, and in parallel implement these ideas on a simple, classical concrete example. 

\bigskip

Our general problem will be a nonlinear partial differential equation
of the form 
\begin{equation}\label{gen-eq}
u_t = N(u), \qquad u(0) = u_0,   
\end{equation}
i.e. a first order system in time, where we think of $u$ as a scalar or a vector valued function belonging to a scale of either real or complex Sobolev spaces. This scale will be chosen to be $H^s:=H^s(\R^n)$ for the purpose of this discussion, though in practice it often has to be adapted
to the class of problems to be considered. The nonlinearity $N$ represents a nonlinear function 
of $u$ and its derivatives,
\[
N(u) = N (\{ \partial^\alpha u\}_{|\alpha| \leq k}) ,
\]
where we will refer to $k$ as the order of the evolution. Here typical examples include $k=1$ (hyperbolic equations), $k = 2$ (Schr\"odinger 
type evolutions) and $k=3$ (KdV type evolutions). But many other 
situations arise in models which are nonlocal, e.g. in water waves 
one encounters $k = \frac12$ for gravity waves respectively $k = \frac32$ for capillary waves.  

Some problems are most naturally formulated as second order evolutions in time, for instance nonlinear wave equations. While some such problems admit also good first order in time formulations (e.g. the compressible Euler flow), it is sometimes better to treat them as second order. 
Regardless, our road-map still applies, with obvious adjustments.

\bigskip

Our model problem will be a classical first order symmetric hyperbolic system
in $\R \times \R^n$, of the form
\begin{equation}\label{sym-hyp}
\partial_t  u = \cA^j(u) \partial_j u, \qquad u(0) = u_0  ,  
\end{equation}
where $u$ takes values in $\R^m$ and the $m\times m $ matrices $\cA^j$ are symmetric, and smooth as functions of $u$. Here the order of the nonlinearity $N$ is $k = 1$, and the scale of Sobolev spaces to be used is indeed the Sobolev scale.

\bigskip

\subsection{What is well-posedness ?} 
To set the expectations for our problems, we recall the classical Hadamard standard for well-posedness,
formulated relative to our chosen scale of spaces.

\begin{definition}
The problem \eqref{gen-eq} is locally well-posed in a Sobolev space 
$H^s (\mathbb{R}^n)$ if the following properties are satisfied:
\begin{description}
\item[(i) Existence] For each $u_0 \in H^s$ there exists some time $T > 0$ and  a solution $u \in C([0,T]; H^s)$.

\item[(ii) Uniqueness] The above solution is unique in $C([0,T]; H^s)$.

\item[(iii) Continuous dependence] The data to solution map is continuous from $u_0 \in H^s$ to $u \in C([0,T];H^s)$.
\end{description}
\end{definition}

As a historical remark, we note that Hadamard primarily 
discussed the question of well-posedness in the context of 
linear pde's, specifically for the Laplace and wave equation,
beginning with an incipient form in \cite{hadamard}, 
and a more developed form in \cite{hadamard1923lectures}.
It is in the latter reference where the continuous dependence 
is discussed, seemingly inspired by Cauchy's theorem for ode's.

\medskip

The above definition should not be taken as universal, but rather as a good starting point, which may need to be adjusted 
depending on the problem. Consider for instance the uniqueness
statement, which, as given in (ii), is in the strongest form, which is often referred to as \emph{unconditional uniqueness}.
Often this may need to be relaxed somewhat, particularly when 
low regularity solutions are concerned. Some common variations concerning uniqueness are as follows:

\begin{enumerate}[label=\alph*)]
    \item The solutions $u$ in (i) are shown to belong to a smaller space, $X^s_T \subset C([0,T]; H^s(\mathbb{R}^n))$, and then the uniqueness in (ii) holds in the same class.
   
   \item Unconditional uniqueness holds apriori only in a more regular class 
   $H^{N}$ with $N > s$, but the data to solution map extends continuously as a map from $H^s$ to $C([0,T]; H^s)$.
\end{enumerate}

Since we are discussing nonlinear equations here, the lifespan 
of the solutions need not be infinite, i.e.  there is always
the possibility that solutions may blow up in finite time.
In particular, in the context of well-posed problems it is
natural to consider the notion of \emph{maximal lifespan},
which is the largest $T$ for 
which the solution exists in 
$C([0,T); H^s)$; here the limit of $u(t)$ as $t$ approaches
$T$ cannot exist, or else the solution $u$ may be continued further. 

In this context, the last property in the definition should be interpreted to mean in particular that, for a solution $u\in C([0,T];H^s)$, small perturbations of the initial data $u_0$ yield solutions which are also defined in $[0,T]$. This in turn  implies that the maximal lifespan $T = T(u_0)$ is lower semicontinuous as a function of $u_0 \in H^s(\mathbb{R}^n)$.

In view of the above discussion, it is always interesting to provide more precise assertions about the lifespan of solutions,
or, equivalently, continuation (or blow-up) criteria for the solutions. Some interesting examples are as follows:

\begin{enumerate}[label=\alph*)]
    \item The lifespan $T(u_0)$ is bounded from below 
uniformly for data in a bounded set, 
\[
T(u_0) \geq C(\|u_0\|_{H^s}) > 0.
\]
This implies a blow-up criteria as follows:
\[
\lim_{t \to T(u_0)} \| u(t) \|_{H^s} = \infty.
\]
    
\item The blow-up may be characterized in terms of weaker bounds,
\[
\lim_{t \to T(u_0)} \| u(t) \|_{Y} = \infty.
\]
relative to a Banach topology $Y \supset H^s$, or perhaps a time integrated version thereof
\[
\int_{0}^{T(u_0)}  \| u(t) \|_{Y} dt = \infty.
\]
\end{enumerate}

To conclude our discussion of the above definition,  we note that many well-posedness statements also provide additional properties for the flow:

\begin{description}
    \item[Higher regularity] if the initial data has more regularity $u_0 \in H^\sigma$
    with $\sigma > s$, then this regularity carries over to the solution, $u \in C[(0,T);H^\sigma]$, with bounds and lifespan bounds depending only on the $H^s$ size of the data.
    \item[Weak Lipschitz bounds] on bounded sets in $H^s$, the flow is Lipschitz in a weaker topology (e.g. up to $H^{s-1}$ in our model problem).
\end{description}
Both of these properties are often an integral part of a complete theory, and frequently also serve as 
intermediate steps in establishing the main well-posedness result.

In all of the above discussion, a common denominator remains the fact that the solution to data map is 
locally continuous, but not uniformly continuous. It is very natural indeed to redefine 
(expand) the notion of quasilinear evolution equations to include all flows which share
this property.

\bigskip

In many problems of this type, one is interested not only in local well-posedness
in some Sobolev space $H^s$, but also in lowering the exponent $s$ as much as possible.
We will refer to such solutions as rough solutions. Then, a natural question is what kind of regularity thresholds should one expect or aim for in such problems? One clue in this direction comes from the scaling symmetry, whenever available. As an example, our model problem 
exhibits the scaling symmetry
\[
u(t,x) \to u(\lambda t, \lambda x), \qquad \lambda > 0.
\]
The scale invariant initial data Sobolev space corresponding to this symmetry is
the homogeneous space $\dot H^{s_c}$, where $s_c = n/2$. This space is called the critical Sobolev space, and should heuristically be thought of as an absolute lower bound for any reasonable well-posedness result. Whereas in some semilinear dispersive evolutions
one can actually reach this threshold, in nonlinear flows it seems to be 
out of reach in general.

\subsection{ A set of results for the model problem}
In order to state the results, we begin with a discussion of \emph{control parameters}. We will use two such control parameters. The first one is
\[
A = \| u \|_{L^\infty}.
\]
This is a scale invariant quantity, which appears in the implicit constants in all of our bounds.
Our second control parameter is
\[
B = \| \nabla u \|_{L^\infty},
\]
which instead will be shown to control the energy growth in all the energy estimates. Precisely, the norm $B$, plays the role of the $Y$ norm mentioned in the discussion above.

\bigskip

The primary well-posedness result for the model problem is as follows:

\begin{theorem}\label{t:lwp}
The equation \eqref{sym-hyp} is locally well-posed in $H^s$  in the Hadamard sense for $s > \frac{d}{2}+1$.
\end{theorem}

The reader will notice that this result is one derivative above scaling. It is also optimal in some cases, including the scalar case (where the problem can be solved locally using the method of characteristics), but not optimal in many other cases
where the system is dispersive. 

For the uniqueness result we have in effect a stronger statement, which only requires Lipschitz bounds for $u$. This however does not improve the scaling comparison relative to the critical spaces:

\begin{theorem}\label{t:unique}
Uniqueness holds in the Lipschitz class, and we have the $L^2$ difference bound
\begin{equation}
\| (u_1- u_2)(t)\|_{L^2} \lesssim e^{C(A) \int _0^t B(s)\, ds}    \| (u_1- u_2)(0)\|_{L^2}.
\end{equation}
\end{theorem}
This is exactly the kind of weak Lipschitz bound discussed earlier. With a bit of additional effort, for the $H^s$ solutions in Theorem~\ref{t:lwp} this may be extended to a larger range of Sobolev spaces,  
\begin{equation}
\| (u_1- u_2)\|_{L^\infty( [0,T];H^\sigma )} \lesssim  \| (u_1- u_2)(0)\|_{H^\sigma}, \qquad |\sigma| \leq s-1.
\end{equation}
The small price to pay here is that now the implicit constant 
in the estimate depends not only on $A$ and $B$ but also on the norms of $u_1$ and $u_2$ in $C([0,T];H^s)$.

\bigskip

A key role in the proof of the well-posedness result is played by the energy estimates,
which are also of independent interest:

\begin{theorem}\label{t:ee}
The following bounds hold for for solutions to \eqref{sym-hyp} for all $s \geq  0$:
\begin{equation}
\| u(t)\|_{H^s} \lesssim e^{C(A) \int_0^t B(s)\, ds}    \| u(0)\|_{H^s}.
\end{equation}
\end{theorem}

Finally, as a corollary of the last result, we obtain a continuation criteria 
for solutions:

\begin{theorem}\label{t:lifespan}
Solutions can be continued  in $H^s$ for as long as $\int B$ 
remains finite.
\end{theorem}

Theorem~\ref{t:lwp} has been first proved by Kato~\cite{Kato}, borrowing ideas from nonlinear semigroup theory, see e.g. Barbu's book \cite{Barbu}. The existence and uniqueness part, as well
as the energy estimates, can also be found in standard references, e.g. in the books of 
Taylor~\cite{Taylor}, H\"{o}rmander~\cite{Hormander} and Sogge~\cite{Sogge} (in the last two
the wave equation is considered, but the idea is similar). However, interestingly enough, the 
continuous dependence part is missing in all these references. We did find presentations of continuous dependence arguments inspired from Kato's work in Chemin's book \cite{Chemin}, and also on Tao's blog \cite{Tao}.

Our objective for remainder of the paper will be to provide complete proofs for the four theorems above, which the reader may take as a guide 
for his problem of choice. While these results are not new in the model case we consider, to the best of our knowledge this is the first time when the proofs of these results are presented in this manner.
Along the way, we will also provide extensive comments and pointers
to alternative methods developed along the years.

In particular, we would emphasize the frequency envelope approach 
for the regularization and continuous dependence parts, as well as the time discretization approach for the existence proof.
The frequency envelope approach has been repeatedly used by the authors, jointly with  different collaborators, in a number of papers, see e.g. \cite{st-nlw}, \cite{tat-wm},\cite{mmt}, \cite{HIT}, \cite{no-rel}, with some of the ideas crystalizing along the way. The version of the existence proof based on a time discretization is in some sense very classical, going back to ideas which have originally appeared in the context of semigroup theory; however, its implementation is inspired
from the authors' recent work \cite{no-rel}, though the situation considered here is considerably simpler.

\subsection{An outline of these notes}
Our strategy will be, in each section, to provide some ideas and a broader discussion in the context of the general equation \eqref{gen-eq}, and then show how this works in detail in the context  of our chosen example \eqref{sym-hyp}.

In the next section we introduce the paradifferential form of our equations,
both the main equation and its linearization. This is an idea that goes back to work of 
Bony~\cite{Bony}, and helps clarify the roles played by different frequency interaction modes in the equation. Another very useful reference here is Metivier's more recent book \cite{Metivier}.

Section~\ref{s:energy} is devoted to the energy estimates, in multiple contexts. These are presented both for the full equation, for its linearization, for its associated linear paradifferential flow,
and for differences of solutions. The latter, in turn, yields 
the uniqueness part of the well-posedness theorem. 
 A common misconception here has been that for well-posedness
 it suffices to prove energy estimates for the full equation.
 Instead, in our presentation we regard 
the bound for the linearized problem as fundamental, though, at the 
implementation level, it is the paradifferential flow bound which
can be found at the core.

Section~\ref{s:existence} provides two approaches for the existence part of the well-posedness theorem. The first one, more classical, is  based on an iteration scheme, which works well on our model problem but may run into implementation issues in more complex problems.  The second approach, which we regard as more robust, relies on a time discretization, and 
is somewhat related to nonlinear semigroup theory, which also inspired Kato's work. Two other possible strategies, which have played a role historically, are briefly outlined. 

Section~\ref{s:rough} introduces Tao's notion of frequency envelopes (see for example \cite{tao2000global}), which 
is very well suited to track the flow of energy as time progresses. This is 
then used to show how rough solutions can be obtained as uniform limits of smooth solutions. This is a key step in many well-posedness arguments, and helps decouple the regularity for the initial existence result from the rough data results. 

Finally the last section of the paper is devoted to the continuous dependence result, where we provide the modern, frequency envelopes based approach. At the same time, for a clean, elegant reinterpretation of 
Kato's original strategy we refer the reader to Tao's blog \cite{Tao}.

\subsection{Acknowledgements}
The first author was supported by a Luce Professorship, by the Sloan Foundation, and by an NSF CAREER grant DMS-1845037. The second author was supported by the NSF grant DMS-1800294 as well as by a Simons Investigator grant from the Simons Foundation. Both authors are extremely grateful to MSRI for their full support in holding the graduate  summer school ``Introduction to water waves''  in a virtual format due to the less than ideal circumstances.

\section{A menagerie of related equations}

While ultimately one would want all the results stated in terms of the full nonlinear equation, any successful approach to quasilinear problems 
needs to also consider a succession of closely related linear equations,
as well as associated reformulations of the nonlinear flow. Here we aim to motivate and describe these related flows, stripping away technicalities.

\subsection{The linearized equation} 

This plays a key role in comparing different solutions;
we will write it in the form
\begin{equation}\label{gen-eq-lin}
v_t = DN(u)v, \qquad v(0) = v_0 ,   
\end{equation}
where $DN$ stands for the differential of $N$, which in our setting
is a partial differential operator of order $k$. One may also
reinterpret the equation for the difference of two solutions as a perturbed linearized equation with a quadratic source term;
some caution is required here, because often some structure is lost in doing this, and the question is whether that is not too much.

In the particular case of \eqref{sym-hyp}, the linearized 
equation takes the form
\begin{equation}\label{sym-hyp-lin}
\partial_t  v = \cA^j(u) \partial_j v + D\cA^j(u) v \, 
\partial_j u, \qquad v(0) = v_0   . 
\end{equation}

\subsection{The linear paradifferential equation}

One distinguishing feature of quasilinear evolutions is that 
the nonlinearity cannot be interpreted as perturbative. Nevertheless,
one may seek to separate parts of the nonlinearity which can be seen as 
perturbative, at least at high regularity, in order to better 
isolate and understand the nonperturbative part.

To narrow things down, consider a nonlinear term which is quadratic,
say of the form $\partial^\alpha u_1 \partial^\beta u_2$, 
and consider the three modes of interaction between these terms,
according to the Littlewood-Paley trichotomy,
or paraproduct decomposition,
\[
\partial^\alpha u_1 \partial^\beta u_2 =
T_{\partial^\alpha u_1} \partial^\beta u_2 + T_{\partial^\beta u_2} \partial^\alpha u_1
+ \Pi(\partial^\alpha u_1, \partial^\beta u_2),
\]
where the three terms represent the $low$-$high$, $high$-$low$ respectively the $high$-$high$ frequency interactions.
The high-high interactions in the last term are always perturbative at high regularity,
so are placed into the perturbative box. But one cannot do the same with the low-high or high-low interactions, which are kept on the nonperturbative side. This is 
closely related to the linearization, and indeed, at the end of the day,
we are left with a paradifferential style nonperturbative part of our evolution,
which we can formally write as
\begin{equation}\label{gen-eq-para}
w_t = T_{DN(u)} w, \qquad w(0) = w_0  .  
\end{equation}
Here, one can naively use Bony's notion of paraproduct \cite{Bony} to define the linear 
operator $T_{DN(u)}$ as 
\[
T_{DN(u)} w = \sum_{|\alpha|\leq k } T_{\partial_{p^\alpha}N(u) } \partial^\alpha w,
\]
where $p^\alpha$ is a placeholder for the $\partial^\alpha u$ argument of the nonlinearity $N$.
However there are also other related choices one can make, see for instance the discussion at the end of this subsection. For a discussion on the use of paradifferential calculus in nonlinear PDE's (though not the above notation) 
we refer the reader to Metivier's book \cite{Metivier}.

One can think of the above evolution as a linear evolution of high frequency waves on a low frequency background. Then one can interpret solving 
the nonperturbative part of our evolution as an infinite dimensional 
triangular system, where each dyadic frequency of the solution is 
obtained at some step by solving a linear system with coefficients depending only on the lower components, and in turn it affects the coefficients of the equations for the higher frequency components.
Of course, this should only be understood in a philosophical sense, 
because a variable coefficient flow in general does not preserve frequency localizations. This can sometimes be achieved with careful choices of the paraproduct quantizations, but it never seems worthwhile
to  implement, as the perturbative terms will mix frequencies anyway and add tails.
\medskip

Turning to our model problem, in a direct interpretation the associated paradifferential equation will have the form 
\begin{equation}\label{sym-hyp-para}
\partial_t  w = T_{\cA^j(u)} \partial_j w + T_{D\cA^j(u) \partial_j u}  w \,
, \qquad w(0) = w_0   . 
\end{equation}
However, upon closer examination one may see several choices that could be made. Considering for instance the first paraproduct, which of the following expressions would make the better choice at frequency $2^k$ ?
\[
\cA^j(u)_{< k-8} \partial_j w_k, 
\qquad \cA^j(u_{< k-8}) \partial_j w_k, 
\qquad [\cA^j(u_{< k-8})]_{<k-4} \partial_j w_k. 
\]
The last one may seem the most complicated, but it is also the most accurate. In many cases, including our model problem, it makes no difference in practice. However, one should be aware 
that often a simpler choice, which is made for convenience in 
one problem, might not work anymore in a more complex setting.

\begin{remark}
Here the \emph{frequency gap}, which was set to be equal to $8$ in the above formulas, is chosen rather arbitrarily; 
its role is simply to enforce  the frequency separation between the coefficients and the leading term. On occasion, particularly in large data problems, it is also useful to work instead with a large frequency gap as a proxy for smallness, see e.g. \cite{ST}.
\end{remark}

\subsection{ The paradifferential formulation of the main equations}

Consider first our general equation \eqref{gen-eq}, which we 
can write in the form
\begin{equation}\label{gen-para}
u_t = T_{DN(u)} u + F(u), \qquad u(0) = u_0 .   
\end{equation}
Here one would hope that the paradifferential source term can be seen as perturbative, in the sense that
\[
F: H^s \to H^s, \qquad \text{   Lipschitz}.
\]

Similarly we can write the linearized equation \eqref{gen-eq-lin} in the same format,
\begin{equation}\label{gen-para-lin}
v_t = T_{DN(u)} v + F^{lin}(u) v, \qquad v(0) = v_0 ,   
\end{equation}
with the appropriate nonlinearity $F^{lin}$. This is still based on the paradifferential
equation \eqref{gen-para}, but can no longer be interpreted as the direct paralinearization 
of the linearized equation. This is because the expression $F^{lin}(u) v$ also contains 
some low-high interactions, precisely those where $v$ is the low frequency factor.

\section{Energy estimates}
\label{s:energy}

Energy estimates are a critical part of any well-posedness result, even if they do not tell the entire story. In this section we begin with a heuristic discussion of several ideas in the general case, and then continue with some more concrete analysis in the model case. 

\subsection{The general case}
Consider first the energy estimates for the general problem \eqref{gen-eq},
where it is simpler to think of this in the paradifferential formulation \eqref{gen-eq-para}. An energy estimate for this problem is an estimate that allows us to
control the time evolution of the Sobolev norms of the solution. 
In the simplest formulation, the idea would be to prove that 
\[
\frac{d}{dt} \| u \|^2_{H^\sigma} \lesssim C \| u\|^2_{H^\sigma},
\]
with a constant $C$ that at the very least depends on the $H^s$ norm of $u$.

There are two points that one should take into account when considering such estimates.
The first is that it is often useful to strenghten such bounds by relaxing the dependence of the constant $C$ on $u$. Heuristically, the idea is that this constant measures
the effect of nonlinear interactions, which are strongest when our functions are pointwise large, not only large in an $L^2$ sense. Thus, it is often possible 
to replace the constant $C$ with an analogue of the uniform control norm $B$ in the 
model case, perhaps with some additional implicit dependence on another 
scale invariant uniform control parameter $A$.  See however the discussion in Remark~\ref{r:high-k}.

A second point is that, although it is tempting to try to work directly with the 
$H^s$ norm, it is often the case that the straight $H^s$ norm is not well adapted to the structure of the problem; see e.g. what happens in water waves \cite{abz}, \cite{HIT}.
Then it is useful to construct energy functionals $E^\sigma$ adapted to the problem at hand. 
For these energies we should aim for the following properties

\begin{enumerate}[label=\roman*)]
    \item Energy equivalence:
    \begin{equation}
    E^\sigma(u) \approx \|u\|_{H^\sigma}^2.
    \end{equation}
    
\item Energy propagation
\begin{equation} \label{En-sigma}
   \frac{d}{dt} E^\sigma(u) \lesssim_A  B \| u\|^2_{H^\sigma}, 
\end{equation}
where the control parameter $B$ satisfies
\begin{equation}
    B \lesssim \| u\|_{H^s} .
\end{equation}
\end{enumerate}

Now consider our main equation written in the form \eqref{gen-eq-para}. 
For the perturbative part of the nonlinearity $F$ we hope to have some boundedness,
\begin{equation}\label{F-bound}
\| F(u) \|_{H^\sigma} \lesssim_A B \| u\|_{H^\sigma}   .  
\end{equation}
This in turn allows us to reduce 
nonlinear energy bounds of the form \eqref{En-sigma}
to similar bounds for the linear paradifferential equation \eqref{gen-para}.
One may legitimately worry here 
that some structure is lost when we decouple the 
paradifferential coefficients from the evolution variable; however, the point is that these two objects are indeed separate, as they represent different frequencies of the solution.

\begin{remark}
In our discussion here we took the simplified view that bounds for $F$ begin at $\sigma = 0$. But this is not always the case in practice, and often one needs to identify 
the lower range for $\sigma$ where this works; see e.g. the 
nonlinear wave equation \cite{st-nlw}, the
wave map equation \cite{tat-wm}, or the water wave problem considered in \cite{ait}.
\end{remark}

Now consider the paradifferential evolution \eqref{gen-para}, and begin with the $L^2$ case by setting 
$\sigma = 0$.  Then we need to produce a linearized type energy $E^{0,lin}_u$
so that the solutions  satisfy 
\begin{equation}\label{En-0-para}
   \frac{d}{dt} E^{0,lin}_u(w) \lesssim_A  B \| w\|^2_{L^2}   .
\end{equation}
Then the associated nonlinear energy at $\sigma=0$ would be 
\[
E^0(u) = E^{0,lin}_u (u).
\]
If $E_u^{0,lin}(w)= \|w\|_{L^2}^2$, then the bound \eqref{En-0-para} would simply require that the paradifferential operator $T_{DN(u)}$ is  essentially antisymmetric in $L^2$.
If that is not true, then the backup plan is to find an equivalent Hilbert norm on $L^2$
so that the antisymmetry holds. Some care is however needed; if this norm depends on $u$, then this dependence needs to be mild.

The next step is to consider a larger $\sigma$. By interpolation it suffices to work with integer $\sigma$, in which case one might simply differentiate \eqref{gen-eq-para},
\[
(\partial^\sigma  w)_t = T_{DN(u)} (\partial^\sigma w) + [\partial^\sigma, T_{DN(u)}] w.
\]
Here we would be done if the last commutator is bounded from $H^\sigma$ into $L^2$.
In principle that would be the case almost automatically at least when the order $k$ of $N$ is at most one. One can heuristically associate this with the finite speed of propagation in the high frequency limit. 

\begin{remark}
\label{r:high-k}
The case $k > 1$, which corresponds to an infinite speed of propagation, 
is often more delicate; see e.g. \cite{mmt,mmt1,mmt2} for quasilinear Schr\"odinger flows, or \cite{it-cap} for capillary waves. There
one needs to further develop the function space structure based on either 
dispersive properties of solutions, or on normal forms analysis.
\end{remark}

\bigskip 

\subsection{Coifman-Meyer and Moser type estimates}
Before considering our model problem, we briefly review some standard bilinear and nonlinear 
estimates that play a role later on. In the context of bilinear estimates, a standard tool is to consider the Littlewood-Paley paraproduct type decomposition of the product of two functions, which leads to  Coifman-Meyer type estimates, see \cite{MR518170}, \cite{MR3052499}:

\begin{proposition} Using the standard paraproduct notations, one has the following estimates
\begin{equation}\label{CM}
\begin{aligned}
&\Vert T_fg\Vert_{L^2} \lesssim \Vert f\Vert_{L^\infty} \Vert g\Vert_{L^2},\\
&\Vert T_fg\Vert_{L^2} \lesssim \Vert g\Vert_{BMO} \Vert f\Vert_{L^2},\\
&\Vert \Pi (f,g)\Vert_{L^2} \lesssim \Vert f\Vert_{BMO} \Vert g\Vert_{L^2},\\
\end{aligned}
\end{equation}
as well as the  commutator bound
\begin{equation}
    \Vert [P_k, f]g\Vert_{L^2} \lesssim 2^{-k}\Vert \partial_xf\Vert_{L^{\infty}}\Vert g\Vert_{L^2}.
\end{equation}
Here $P_k$ is the Littlewood-Paley projection onto frequencies $\approx 2^k$.
\end{proposition}
These results are standard in the harmonic/microlocal analysis community.  For nonlinear expressions we use
Moser type estimates instead:
\begin{proposition}
The following Moser estimate holds for a smooth function $F$, with $F(0)=0$, and $s \geq 0$:
\[
\Vert F(u)\Vert_{H^s}\lesssim_{\Vert u\Vert_{L^{\infty}}}\Vert u\Vert_{H^s}.
\]
\end{proposition}
Of course many more extensions of both the bilinear and the nonlinear estimates above are available.

\subsection{The model case}
We now turn our attention to our model problem, where, if we adopt the expression \eqref{sym-hyp-para} for the paradifferential flow, the source term $F(u)$ is given by \begin{equation}
F(u) = \cA^j \partial_j u -   T_{\cA^j(u)} \partial_j u - T_{D\cA^j(u) \partial_j u}   u .
\end{equation} 
We can rewrite this in the form
\begin{equation}\label{F(u)}
F(u) =   \Pi(\cA^j(u), \partial_j u) + 
T_{\partial_j u} \cA^j(u) - T_{D\cA^j(u) \partial_j u}   u .
\end{equation}
For this expression we can show that it always plays a perturbative role:

\begin{proposition}\label{p:F}
The above nonlinearity $F$ satisfies the following bounds:

i) Sobolev bounds 
\begin{equation}
\| F(u) \|_{H^\sigma} \lesssim_A B \|u\|_{H^\sigma}, \qquad \sigma \geq 0.    
\end{equation}

ii) Difference bounds
\begin{equation}
\| F(u) - F(v) \|_{H^\sigma} \lesssim_A B \left[\|u-v\|_{H^\sigma}+
 \| u-v\|_{L^\infty} (\| u\|_{H^\sigma} + \|v\|_{H^\sigma}) \right]
, \qquad \sigma \geq 0   , 
\end{equation}
as well as 
\begin{equation}
\| F(u) - F(v) \|_{L^2} \lesssim_A B \|u-v\|_{L^2}.
\end{equation}
\end{proposition}
The next to last bound shows in particular that $F$ is Lipschitz in $H^s$
for $s > d/2$. The simplification in the case $\sigma = 0$ is also useful in order to bound differences of solutions in the $L^2$ topology.

\begin{proof}
i) We use the expression \eqref{F(u)} for $F$. The first term can be estimated using a version of the Coifman-Meyer estimates and Moser estimates by
\[
\| \Pi(\cA^j(u), \partial_j u)\|_{H^\sigma} \lesssim \| \cA^j(u)\|_{H^\sigma} \| \partial_j u\|_{BMO} \lesssim_A B \|u\|_{H^\sigma}.
\]
For the second term we use again paraproduct bounds and Moser estimates to get
\[
\| T_{\partial_j u} \cA^j(u) \|_{H^\sigma} \lesssim \| \partial_j u\|_{L^\infty}
\|  \cA^j(u) \|_{H^\sigma} \lesssim_A  \| \partial_j u\|_{L^\infty}
\|  u \|_{H^\sigma}.
\]
The third term is similar to the second.
\bigskip

ii) First, we note the representation 
\[
\cA(u) - \cA(v) =: G(u,v)(u-v),
\]
which we use to separate $u-v$ factors. Here $G(u,v)$ is a smooth function of $u$ and $v$. Then taking differences 
in the first term of $F$, we need two estimates
\[
\| \Pi(\cA^j(u), \partial_j (u-v))\|_{H^\sigma} \lesssim \| \partial \cA^j(u)\|_{L^\infty} \|  u-v\|_{H^\sigma} \lesssim_A B \|u-v\|_{H^\sigma}
\]
respectively 
\[
\| \Pi(G(u,v)(u-v), \partial_j v)\|_{H^\sigma} \lesssim 
\| G(u,v)(u-v)\|_{H^\sigma} \| \partial v\|_{L^\infty} \lesssim_A B 
(\|u-v\|_{H^\sigma} + \| u-v\|_{L^\infty}(\|u\|_{H^\sigma} +\|v\|_{H^\sigma})),
\]
noting that for $\sigma=0$ the last term can be avoided.

Similarly we have two estimates corresponding to the second term in $F$, namely
\[
\begin{aligned}
\|T_{\partial_j u} \cA^j(u) - T_{\partial_j v} \cA^j(v)\|_{H^\sigma} &=  \|T_{\partial_j u}[G(u,v) (u-v)] - T_{\left[\partial_j u-\partial_j v\right]} \cA^j(v)\|_{H^\sigma}\\
&\lesssim  \|T_{\partial_j u}[G(u,v) (u-v)]\|_{H^\sigma} + \|T_{\left[\partial_j u-\partial_j v\right]} \cA^j(v)\|_{H^\sigma}, \\
\end{aligned}
\]
where
\[
\|T_{\left[\partial_j u-\partial_j v\right]} \cA^j(v)\|_{H^\sigma} \lesssim \| u-v\|_{L^\infty}
\|  \partial_j \cA^j(v) \|_{H^\sigma} \lesssim_A  B \|u-v\|_{L^\infty} 
\|  v \|_{H^\sigma},
\]
respectively
\[
\|T_{\partial_j u} [G(u,v)(u-v)] \|_{H^\sigma} \lesssim_A  
\| \partial_j u\|_{L^\infty} (\|u-v\|_{H^\sigma} + \| u-v\|_{L^\infty}(\|u\|_{H^\sigma} +\|v\|_{H^\sigma})),
\]
both with obvious simplifications if $\sigma = 0$. Finally,  the bounds for the third term in $F$ are similar to the ones for the second.

\end{proof}

\begin{remark}
For this Proposition~\ref{p:F} one can further relax $B$ to a $BMO$ norm, 
\[
B = \| \nabla u \|_{BMO}.
\]
On the other hand we can also simplify the paradifferential equation \eqref{sym-hyp-para}
to a simpler version
\[
w_t = T_{\cA^j(u)} \partial_j w ,
\]
but in this case  we no longer can relax $B$ to a BMO norm.

\end{remark}

Next we consider the paradifferential equation:

\begin{proposition}\label{p:para}
Assume that $u \in L_{t,x}^\infty$ and $\nabla u \in L^1_t L_x^\infty$ (i.e. $B\in L^1_t$).
Then the paradifferential equation \eqref{sym-hyp-para} is well-posed in all $H^\sigma$ spaces, $\sigma \in \R$, and 
\begin{equation}
\frac{d}{dt} \| w\|_{H^\sigma}^2 \lesssim_A B \|w\|_{H^\sigma}^2    .
\end{equation}
\end{proposition}

\begin{proof}
We first consider the energy estimate, where we work with the corresponding 
inhomogeneous equation,
\begin{equation}\label{sym-hyp-para-inhom}
\partial_t  w = T_{\cA^j(u)} \partial_j w + T_{D\cA^j(u) \partial_j u}  w  + f\,
, \qquad w(0) = w_0  .  
\end{equation}
The $L^2$ bound is easiest; we have 
\[
\frac12 \frac{d}{dt}  \|w\|_{L^2}^2 = \int w \cdot T_{\cA^j(u)} \partial_j w +
w\cdot T_{D\cA^j(u) \partial_j u} w + w \cdot f \, dx.
\]
In the second term we simply estimate the para-coefficient in $L^\infty$.
In the first term we commute and integrate by parts,
to arrive at 
\[
\frac12 \int - w \cdot T_{\partial_j \cA^j(u)}  w + w \cdot (T_{\cA^j(u)}- (T_{\cA^j(u)})^*) \partial_j w\,  dx ,
\]
where due to the symmetry of the matrices $\cA^j$ we have the bound
\begin{equation}\label{adj-diff}
\| (T_{\cA^j(u)}- (T_{\cA^j(u)})^*) \partial_j w \|_{L^2} \lesssim_A B\|w\|_{L^2},
\end{equation}
which shows that the corresponding paraproduct operators are self-adjoint at 
leading order. Here we use the $^*$ notation to denote the adjoint of an operator. Hence we obtain
\[
\left|\frac{d}{dt} \| w\|_{L^2}^2 \right| \lesssim_A B \|w\|_{L^2}^2 + \| w\|_{L^2} \|f\|_{L^2},
\]
which further by Gronwall's inequality yields
\begin{equation}
\| w\|_{L_t^\infty([0,T]; L_x^2)} \lesssim_{A}e^{\int _0^TB\,dt} (\|w(0)\|_{L_x^2} + \| f \|_{L_t^1 L^2_x})   . 
\end{equation}

This by itself does not prove well-posedness in $L^2$, it only proves uniqueness. However, a similar bound will hold for the backward adjoint system in the same spaces; this is because the adjoint system coincides with the direct system modulo $L^2$ bounded terms. Together, these two pieces of information yield $L^2$ well-posedness for the paradifferential system in $L^2$.
This is a standard linear duality argument, where the solutions 
are constructed by a direct application of the Hahn-Banach Theorem.
In a nutsell, one has the following equivalencies, see for instance 
\cite{Horm}:
\bigskip

\[
\text{Energy estimates for the direct forward problem} \Longleftrightarrow
\text{Existence for the adjoint backward problem}
\]
\[
\text{Energy estimates for the adjoint backward problem} \Longleftrightarrow
\text{Existence for the direct forward problem}
\]
\bigskip

Exactly the same argument applies in $H^\sigma$, with the small change that now the 
the adjoint system should be considered in $H^{-\sigma}$. There the 
bound \eqref{adj-diff} is replaced by 
\begin{equation}\label{adj-diff+}
\|  (\jD^\sigma T_{\cA^j(u)}- (T_{\cA^j(u)})^*\jD^\sigma) \partial_j w \|_{L^2} \lesssim  \| \nabla \cA(u)\|_{L^\infty} \|w\|_{H^\sigma}\lesssim_A B  \|w\|_{H^\sigma}.
\end{equation}

\end{proof}

Combining the last two propositions, Proposition~\ref{p:F} and Proposition~\ref{p:para}, we obtain the $H^\sigma$ bound in 
Theorem~\ref{t:ee}.

\subsection{The linearized equation}
Next, we turn our attention to the linearized equation, 
which we also write in a paradifferential form
\begin{equation}\label{sym-hyp-lin-para}
\partial_t  v = T_{\cA^j(u)} \partial_j v + T_{D\cA^j(u) \partial_j u} v + F^{lin}(u) v,
 \qquad v(0) = v_0  ,  
\end{equation}
where 
\[
F^{lin}(u) v := \Pi(\cA^j(u), \partial_j v ) +  \Pi(D\cA^j(u) \partial_j u, v)
+ T_{\partial_j v} \cA^j(u) + T_v (D\cA^j(u) \partial_j u) := F^{lin}_{\Pi}(u) v 
+ F^{lin}_{T}(u) v .
\]

We note here that the equation \eqref{sym-hyp-lin-para} is not exactly 
a true paralinearization of the linearized equation, as $F^{lin}_{T}(u) v $
does contain low-high interactions. This difference is observed in the estimates satisfied by the two terms. 

On one hand, the term $F^{lin}_{\Pi}(u) v$ satisfies good bounds in all Sobolev spaces,
\begin{equation}
   \| F^{lin}_\Pi(u) v \|_{H^\sigma} \lesssim_A B \|v\|_{H^\sigma}, \qquad \sigma \geq 0 ,
\end{equation}
so it can be seen as a true perturbative term. This a simple, Coifman-Meyer type estimate which is left for the reader.

On the other hand, assuming we know that $u \in H^s$, the term $F^{lin}_{T}(u) v$
can at best be estimated in $H^{s-1}$, and there of course we could not use 
the control norms, instead we would have to use the full $H^s$ norm of $u$.
However, we can use the control norms for  $L^2$ bounds to directly estimate
\begin{equation}
   \| F^{lin}_T(u) v \|_{L^2} \lesssim_A B \|v\|_{L^2}. 
\end{equation}

Combining the last two estimates with Proposition~\ref{p:para} we perturbatively obtain 

\begin{proposition}\label{p:lin-wp}
Assume that $A \in L^\infty$ and that $B \in L^1$. Then the  linearized equation \eqref{sym-hyp-lin} is well-posed in $L^2$, with bounds
\begin{equation}
\frac{d}{dt} \| v\|_{L^2}^2 \lesssim_A B \|v\|_{L^2}^2    .
\end{equation}
\end{proposition}

We observe the obvious fact that one does not need paradifferential calculus in order 
to prove this proposition; a simple integration by parts suffices. 
However, it is instructive to dissect the terms in the equation and understand their 
respective roles. Also, it is interesting to observe that in appropriate 
settings, the linearized equation can be thought of as a perturbation of 
the associated paradifferential equation.

\begin{remark}\label{r:sigma-lin}
Well-posedness and bounds for the linearized equation can be also obtained in all $H^\sigma$ spaces for $|\sigma| \leq s-1$. However, this can no longer be done 
in terms of our control parameters; for instance if $\sigma = s-1$ then we need to use the full $H^s$ norm of the solutions. While interesting, this observation will not be needed for the rest of the paper.
\end{remark}

\subsection{Difference bounds and uniqueness}

The easiest way to compare two solutions $u_1$ and $u_2$ for \eqref{gen-eq} is to subtract their respective equations, to obtain an equation for $v= u_1-u_2$. In the general case, using the form \eqref{gen-para} of the equation, we obtain
\[
v_t = T_{DN(u_1)} v + T_{DN(u_1)-DN(u_2)} u_2 + F(u_1)-F(u_2).
\]
Here we identify this equation as the paradifferential equation associated to  $u_1$, but with two source terms, which  we would like to interpret as perturbative in a low regularity Sobolev space,
say $L^2$. That would yield a bound of the form 
\begin{equation}\label{diff-est}
\| v(t)\|_{L^2} \lesssim e^{C(A)\int_0^t B(s) ds} \|v(0)\|_{L^2},
\end{equation}
where $A= A_1+A_2$, $B = B_1+B_2$, with $A_i=\Vert u_i\Vert_{L^\infty}$, and $B_i=\Vert \nabla u_i\Vert_{L^{\infty}}$, for $i=\overline{1,2}$.

\medskip

Let us see how this works out in our model problem. We will show that
\begin{proposition}\label{p:diff}
Let $u_1$ and $u_2$ be two Lipschitz solutions to \eqref{sym-hyp}
with associated control parameters $A_1,B_1$ respectively $A_2,B_2$.
Then their difference $v = u_1-u_2$ satisfies the bound \eqref{diff-est}.

\end{proposition}

\begin{proof}
We have already seen in Proposition~\ref{p:para} that the paradifferential evolution is well-posed in $L^2$, and in Proposition~\ref{p:F} that we have a good Lipschitz bound for $F$. It remains to bound the remaining  difference
\[
\|T_{DN(u_1)-DN(u_2)} u_2\|_{L^2} \lesssim_A B \|u_1-u_2\|_{L^2}.
\]
For this we write
\[
\begin{aligned}
T_{DN(u_1)-DN(u_2)} u_2 = & \  T_{\cA^j(u_1)-\cA^j(u_2)} \partial_j u_2 
+ T_{D\cA^j(u_1) \partial_j u_1  - D\cA^j(u_2) \partial_j u_2} u_2
\\
 = & \  T_{\cA^j(u_1)-\cA^j(u_2)} \partial_j u_2 
+ T_{(D\cA^j(u_1) - D\cA^j(u_2)) \partial_j u_1} u_2
- T_{\partial_j D\cA^j(u_2) (u_1  -  u_2)} u_2 
\\
& \ + T_{\partial_j(D\cA^j(u_2) (u_1  -  u_2))} u_2  .
\end{aligned}
\]

For the first term we have a Coifman-Meyer type bound
\[
\|T_{\cA(u_1) - \cA(u_2)} \nabla u_2\|_{L^2} \lesssim \| u_1-u_2\|_{L^2}  \| \nabla u_2 \|_{BMO}
\lesssim B  \|u_1-u_2\|_{L^2}.
\]
The second term is even easier,
\[
\| T_{(D\cA^j(u_1) - D\cA^j(u_2)) \partial_j u_1} u_2\|_{L^2}
\lesssim \| (D\cA^j(u_1) - D\cA^j(u_2)) \partial_j u_1\|_{L^2} \|u_2\|_{L^\infty}
\lesssim_A B \|u_1-u_2\|_{L^2},
\]
and the third term is similar. Finally, in the fourth term we can 
use Coifman-Meyer to rebalance again the derivatives and obtain 
\[
\| T_{\partial_j(D\cA^j(u_2) (u_1  -  u_2)} u_2 \|_{L^2} \lesssim \| D\cA^j(u_2) (u_1  -  u_2)\|_{L^2} \| \nabla u_2\|_{BMO},
\]
concluding as before.
\end{proof}

\begin{remark}
The observant reader may have noticed that for our model problem the 
difference bound can be directly proved using a simple integration by 
parts, without any need for paradifferential calculus, and may wonder
why we are doing it this way.  There are three reasons for this:
(i) to show that it works, (ii) to show how both the bound for the full 
equation and the bound for the difference equation can be seen as 
two sides of the same coin, and (iii) to provide a guide for the reader
for situations where a simpler approach does not work.
\end{remark}

\begin{remark}\label{r:sigma-diff}
In the same vein as in Remark~\ref{r:sigma-lin}, bounds for the difference equation can be also obtained in all $H^\sigma$ spaces for $|\sigma| \leq s-1$.
\end{remark}

\begin{remark}
In our particular example it was easy to cast the difference equation 
in a form which is very much like the linearized equation. However, this is not always the case. For this reason, we point out that there is another way one can think of difference bounds, namely by viewing the two initial data $u_{01}$ and $u_{02}$
as being connected via a one parameter family of data $u_{0h}$ where $h \in [1,2]$.
Then we can interpret the difference $u_2-u_1$ as 
\[
u_2 - u_1 = \int_{1}^2 \frac{d}{dh} u_h\, dh,
\]
where $u_h$ are the solutions with data $u_{0h}$. Here the integrand represents 
a solution to the linearized equation around $u_h$. Hence difference bounds for $u_2-u_1$
can be obtained by integrating bounds for the linearized equation.
The only downside to such an argument is that such bounds will require the 
control parameters for the entire family of solutions, rather than just the end-points.

\end{remark}

\section{Existence of solutions}
\label{s:existence}

Here we consider the question of existence of solutions for the evolution
\eqref{gen-eq} with initial data in $H^s$, where $s$ will be taken sufficiently large.
The idea here is to construct a good sequence of approximate 
solutions $u^{n}$, which will eventually be shown to converge in a weaker 
topology. The tricky bit is to choose the correct iteration scheme.

Naively, one might think of trying to base such a scheme on the linearized flow, setting
\[
\partial_t (u^{n+1}-u^n) - DN(u^{n})(u^{n+1}-u^n) = - (\partial_t u^n - N(u^n)), \qquad (u^{n+1}-u^n)(0) = 0,
\]
where the expression on the right represents the error at step $n$. Here one can 
eliminate the time derivative of $u^n$ and rewrite this as
\[
\partial_t u^{n+1} - DN(u^{n})u^{n+1} =  N(u^n) - DN(u^n) u^n, \qquad u^{n+1}(0) = u_0.
\]
This would be akin to a Nash-Moser scheme, which, even when it works, loses
derivatives. That may be reasonable in a small divisor situation, but not so much
if our goal is to obtain a Hadamard style well-posedness result. Nevertheless,
Nash-Moser schemes have been used on occasion to produce solutions for  quasilinear
evolutions, though often they prove to be unnecessary.

\begin{remark}
We observe that for the existence  of solutions one does not need to work from the start at low regularity. As we will see, rough solutions can be constructed later on as limits of smooth solutions. This is strictly speaking not necessary in our model problem, but
for more nonlinear, geometric problems it does seem to make a difference.
This is because in such situations it is often easier to compare exact solutions via
the linearized equation which is a geometric object, instead of working with approximate solutions where the geometric character might be lost.
\end{remark}

We will present two strategies to prove existence, and at the end we 
point out several other methods which have been successfully used in existence proofs.

\subsection{Take 1: an iterative/fixed point construction}
In order not to lose derivatives in the approximation scheme, the idea here
is to carefully choose how to distribute $u^{n+1}$ and $u^n$ in the iteration. 
A key observation is that, whereas solving the linearized equation would cause a loss
of derivatives, solving the paradifferential equation does not in general.
Then, a good starting point would be the formulation \eqref{gen-eq-para} of the equations,
which would suggest the following iteration scheme:
\begin{equation}
\partial_t u^{n+1} - T_{DN(u^{n})} u^{n+1} =  F(u^n), \qquad u^{n+1}(0) = u_0.
\end{equation}
We will apply this scheme on a time interval $[0,T]$, with $T=T(M)$ sufficiently small
depending on the initial data size
\[
M := \|u_0\|_{H^s}.
\]

For the above sequence $u^{n}$ the aim would be to inductively prove two uniform bounds in $[0,T]$:
\begin{equation}
\| u^n \|_{L_t^\infty H_x^s} \leq CM    ,
\end{equation}
and 
\begin{equation}
\| u^{n+1} - u^{n}\|_{L_t^\infty L_x^2} \leq C(M) T \| u^{n} - u^{n-1}\|_{L_t^\infty L_x^2},
\end{equation}
where $C$ is a fixed large constant.
In the last bound, the time interval size $T$ is used in order to gain smallness for the constant, which is needed in order to obtain 
convergence. Together, these two bounds imply convergence in $L_t^\infty L_x^2$ to some function $u$, as well as  $L_t^\infty H_x^s$ regularity for the limit. This in general suffices in order to show that the limit solves the equation.

To obtain uniform bounds for this evolution one would need two pieces of information:
\begin{enumerate}
    \item Well-posedness of the paradifferential equation \eqref{gen-eq-para} in $L^2$ and more generally in 
    all $H^s$ spaces. Heuristically, the two should be equivalent, as the operator 
    $T_{DN(u^{n})}$ does not change the dyadic frequency localization. In practice 
    though it might not be as easy, as leakage to other frequencies may occur, and 
    in particular even the associated Hamilton flow might  not preserve the dyadic localization on a unit time scale.
    \item Lipschitz property of $F$ in Sobolev spaces. More generally, a bound of the form
    \begin{equation}\label{dF}
    \| F(u) - F(v) \|_{H^\sigma} \leq C(\| u\|_{H^s},\|v\|_{H^s}) \| u-v \|_{H^\sigma}, \qquad \sigma \geq 0,   
    \end{equation}
    which should be thought of as a Moser type inequality.
\end{enumerate}

In addition to uniform bounds in a strong norm $H^s$,
one would also like to have convergence in a weaker topology,
say $L^2$ for the purpose of this presentation. The difference 
equation reads
\begin{equation}\label{diff-eqq}
(\partial_t - T_{DN(u^{n})}) (u^{n+1}-u^{n})
=  F(u^n) - F(u^{n-1}) + (T_{DN(u^{n-1})}- T_{DN(u^{n})}) u^n.
\end{equation}
Here energy estimates in $L^2$ would follow from (1) and (2) above,
provided that the last difference has a good bound
\[
\| (T_{DN(u^{n-1})}- T_{DN(u^{n})}) u^n \|_{L^2} \lesssim 
C(\|u^{n-1}\|_{H^s},\|u^n\|_{H^s}) \| u_{n} - u_{n-1} \|_{L^2}.
\]
This is in general relatively straightforward if $s$ is large enough.

\begin{remark}
The argument above yields solutions which are apriori only in $L_t^\infty H_x^s$ 
as  opposed to $C(H^s)$, as desired. Getting continuity in $H^\sigma$ for $\sigma < s$
is relatively straightforward by interpolation, but proving continuity in $H^s$ requires  considerable extra work\footnote{e.g. by showing continuity in time of solutions to the linear paradifferential equation.} if one wants a direct argument. The easy way out is to rely on the arguments in the next section, where we show that all $H^s$ solutions can be  seen as uniform limits
of smooth solutions.
\end{remark}

\begin{remark}
The above iterative argument can be rephrased as a fixed point argument as follows.
For $u \in C[0,T;H^s]$ we define $Lu(t) : = v$ as the solution to 
\[
\partial_t v - T_{DN(u)} v =  F(u), \qquad v(0) = u_0
\]
Then the desired solution $u$ has to be a fixed point for $L$. Solutions to this fixed point problem may often be obtained using the contraction principle in the 
right topology. Precisely, the strategy is to choose the domain of $L$
to be  the ball $B(0,CM)$ in $L^\infty[0,T;H^s]$, but endow this ball with a weaker topology, e.g.  $C[0,T; L^2]$. Then both the mapping properties of $L$ and the small Lipschitz constant can be achieved by choosing the time $T$ sufficiently small. Here for the 
domain we have to choose $L^\infty$ rather than continuity in order
to guarantee completeness.
\end{remark}

\bigskip

We now implement this  scheme for our model problem. 
Denoting $M = \|u_0\|_{H^s}$, we will prove inductively that for fixed 
large enough $T$ and small enough $T$, we have the bound
\[
\| u^n \|_{C(0,T;H^s)} \leq CM.
\]
Taking this as induction hypothesis we have the following bounds for the control 
parameters $A^n$ and $B^n$ associated to $u^n$:
\[
A^n, B^n \lesssim CM.
\]
Then we can estimate $u^{n+1}$ in $H^s$ by combining Proposition~\ref{p:para}
and Proposition~\ref{p:F}
to obtain 
\[
\frac{d}{dt} \| u^{n+1}\|_{H^s}^2 \lesssim C(M)(1 + \| u^{n+1}\|_{H^s}^2) , 
\]
and by Gronwall's inequality we arrive at
\[
\| u^n \|_{C(0,T;H^s)} \lesssim M e^{C(M) T},
\]
with a universal implicit constant.
This completes the induction if we first choose $C$ large enough
(to dominate the implicit constant), and then $T$ small enough 
(depending on $C$ and $M$).

On the other hand, in order to prove the convergence in $L^2$
we use the equation \eqref{diff-eqq} for the difference $u^{n+1}-u^n$,
and claim that the following $L^2$ estimate holds:
\begin{equation}\label{l2-diff}
\frac{d}{dt} \| u^{n+1}-u^n\|_{L^2}^2     \lesssim C(M) \| u^{n+1}-u^n\|_{L^2}^2 
+ C(M) \| u^{n}-u^{n-1}\|_{L^2}^2 .
\end{equation}
Assuming this is true, by Gronwall's inequality we obtain
\[
\| u^{n+1}-u^n\|_{C(0,T;L^2)}     \lesssim C(M) T e^{ C(M)T}  \| u^{n}-u^{n-1}\|_{C(0,T;L^2)} ,
\]
which gives us the small Lipschitz constant if $T$ is sufficiently small, depending only on $M$.

It remains to prove \eqref{l2-diff}. For the paradifferential equation we can use Proposition~\ref{p:para} and for the $F$ difference we can use Proposition~\ref{p:F},
so it remains to examine the last term in \eqref{diff-eqq}, and show that
\[
\|(T_{DN(u^{n-1})}- T_{DN(u^{n})}) u^n\|_{L^2} \lesssim C(M) \|u^{n-1}-u^n\|_{L^2}.
\]
In the case of the model problem the difference on the left reads
\[
T_{\cA^j(u^{n-1}) - \cA^j(u^n)} \partial_j u^n + T_{D\cA^j(u^{n-1}) \partial_j u^{n-1}-
D\cA^j(u^{n}) \partial_j u^{n}} u^n.
\]
For the first term we have the obvious bound
\[
\| T_{{\cA}^j(u^{n-1}) - \cA^j(u^n)} \partial_j u^n\|_{L^2} \lesssim 
\| \cA^j(u^{n-1}) - \cA^j(u^n)\|_{L^2} \|\partial_j u^n\|_{L^\infty} \lesssim C(M)
\| u^{n-1}- u^n\|_{L^2}.
\]
The second term is split into three parts,
\[
T_{(D\cA^j(u^{n-1})-
D\cA^j(u^{n})) \partial_j u^{n}} u^n
-
T_{\partial_j D\cA^j(u^{n-1})( u^{n-1}-
u^{n})} u^n
+
T_{\partial_j [D\cA^j(u^{n-1})( u^{n-1}-
u^{n})]} u^n,
\]
where the first two parts are easy to estimate. A similar bound follows for the third term after we move the derivative onto the high frequency factor, using an estimate of the form
\[
\|T_{\partial f} g\|_{L^2} \lesssim \|f\|_{L^2} \|\partial g\|_{BMO},
\]
which is a corollary of the second bound in \eqref{CM}.

\subsection{Take 2: a time discretization method}

Here the idea is to discretize time at a small scale $\epsilon$,
and to construct approximate discrete solutions
$u^\epsilon(j\epsilon)$ with the following properties:

\begin{enumerate}[label=\roman*)]
    \item Uniform bounds
\begin{equation}\label{onestep-ee}
\| u^\epsilon(j\epsilon) \|_{H^s} \leq CM, \qquad j \ll_M \epsilon^{-1}   ; 
\end{equation}    
\item Approximate solution
\begin{equation}\label{onestep-approx}
\| u^\epsilon((j+1)\epsilon) - u^\epsilon (j\epsilon) - \epsilon N(u^\epsilon(j\epsilon)) \|_{L^2}
\lesssim \epsilon^2.
\end{equation} 
\end{enumerate}
Once this is done, if $s$ is large enough\footnote{For instance in our model case case $s > n/2+1$ suffices.} then it is a relatively straightforward matter
to show that a uniform limit $u$ exists\footnote{Here one may extend 
$u^\epsilon$ to all times by linear interpolation.} on a subsequence as $\epsilon \to 0$, by applying the Arzela-Ascoli theorem.
This works in a time interval $[0,T]$ with $T \ll_M 1$.
By passing to the limit in the above bounds in a weak topology, it follows the limit $u$ solves the equation and has regularity 
\[
u \in L^\infty(0,T;H^s) \cap \text{Lip}(0,T;L^2).
\]

The nice feature of this method is that one really only needs to carry out one single step. Precisely, given $u_0 \in H^s$ with size $M$, and $0 < \epsilon \ll 1 $, one needs to find $u_1$ (which corresponds to $u^\epsilon(\epsilon)$ above) with the following properties:
\begin{enumerate}[label=\roman*)']
    \item Uniform bounds
\begin{equation}\label{single-ee}
\| u_1 \|_{H^s} \leq (1+ C(M) \epsilon) \|u_0\|_{H^s};
\end{equation}    
\item Approximate solution
\begin{equation}\label{single-eqn}
\| u_1 - u_0 - \epsilon N(u_0) \|_{L^2}
\lesssim \epsilon^2.
\end{equation} 
\end{enumerate}
Reiterating this, the bound \eqref{onestep-ee} follows by applying a discrete form of Gronwall's inequality.

\begin{remark}
The $\epsilon^2$ bound in ii)' can be harmlessly replaced by $\epsilon^{1+\delta}$
with a small  constant $\delta > 0$. 
\end{remark}

\begin{remark}
Sometimes the square $H^s$ norm of $u$ is not the correct quantity to propagate in time, and one needs to replace it with  appropriate equivalent energies  $E^s$ in property (ii)'.
\end{remark}

\begin{remark}
The choice of the $L^2$ in (ii)' above was in order to keep the exposition simple. However, sometimes a different topology may be required by the problem, 
see e.g. \cite{tat-wm}, \cite{ait}.
\end{remark}

The remaining question is how to construct the single iterate satisfying properties 
(i)',(ii)' above. The obvious choice
would be Euler's method, which is to set
\[
u_1 = u_0 + \epsilon N(u_0),
\]
but this does not work because it loses derivatives.

Inspired by the nonlinear semigroup theory~\cite{Barbu}, one may choose instead to solve \[
u_1 - \epsilon N(u_1) = u_0.
\]
This idea has potential at least when this is an elliptic equation.  Alternatively one may opt for a paradifferential version
\[
u_1 - \epsilon T_{DN(u_0)} u_1  = u_0 + \epsilon F(u_0),
\]
which has the advantage that one only needs to solve a linear  elliptic equation. However, ellipticity is not guaranteed.

Instead, here we will adopt a two steps approach, which has the advantage that no partial differential equation needs to be solved. Precisely, our steps are as follows:

\bigskip

\emph{STEP 1: Regularization.}
Here we take the initial data $u_0$ and we regularize it on an $\epsilon$ dependent scale.
Precisely, if $k$ is the order of the nonlinearity $N$,  then it is natural to choose the  spatial truncation frequency scale to be $\epsilon^{-\frac1{2k}}$, which corresponds to an order 
$2k$ parabolic regularization; this regularization scale is needed in order to be able to bound the error in the Euler step. Then our regularization $\tu$ would
have the following properties:

\begin{enumerate}[label=(\alph*)]
\item Regularization
\begin{equation}\label{tu-high}
   \| \tu \|_{H^{s+k}} \lesssim  \epsilon^{-\frac12} \|u_0\|_{H^s}. 
\end{equation}
\item Energy bound
\begin{equation} \label{tu-ee}
E^s(\tu)  \leq (1+ C(M) \epsilon) E^s(u_0).
\end{equation}    
\item Approximate solution
\begin{equation}\label{tu-err}
\| \tu - u_0  \|_{L^2}
\lesssim \epsilon^2.
\end{equation} 
\end{enumerate}

\bigskip

\emph{STEP 2: Euler iteration.} Here we simply set 
\begin{equation}\label{u1-def}
u_1 = \tu + \epsilon N(\tu)   , 
\end{equation}
so that the approximate solution bound \eqref{single-eqn} becomes relatively straightforward, and the energy bound \eqref{single-ee} becomes akin to proving the energy estimate; see the example below.

\bigskip

We now implement the above strategy on our chosen model problem.
Here our chosen energy is simply the Sobolev norm,
\[
E^N(u) = \| u\|_{H^N}^2.
\]
Our equation has order $k=1$, so the proper regularization scale is $\delta x = \epsilon^\frac12$. Hence, we use a Littlewood-Paley projector to simply define 
\[
\tilde u = P_{<\epsilon^{-\frac12}} u,
\]
and the three properties (a), (b) and (c) above are trivially satisfied.

Next we turn our attention to the Euler iteration \eqref{u1-def}
for which we need to establish the properties (i)' and (ii)'.
We begin with (i)', where it suffices to compare the energies
of $u_1$ and $\tu$. For $|\alpha| \leq N$ we have  
\[
\partial^\alpha u_1 = \partial^\alpha \tilde u + \epsilon \partial^\alpha (\cA^j(\tu) \partial_j \tu).
\]
If $|\alpha| < N$, then in the second term on the right we have at most $N$ derivatives,
so this term has size $O(\epsilon)$ in the $L^2$ norm
\[
\| \partial^\alpha (\cA^j(\tu) \partial_j \tu) \|_{L^2} \lesssim_A \| \tu\|_{H^N},
\]
and we can neglect it.

It remains to consider $|\alpha|=N$. Then we can separate the terms with no more than $N$
derivatives and estimate them as above, using appropriate interpolation inequalities,
\[
 \partial^\alpha (\cA^j(\tu) \partial_j \tu ) =  \cA^j(\tu) \partial^\alpha \partial_j \tu
 + O_{L^2}(B \| \tu\|_{H^N}).
\]
Hence we have 
\[
\partial^\alpha u_1 = \partial^\alpha \tilde u + \epsilon \cA^j(\tu) \partial^\alpha \partial_j \tu + O_{L^2}(\epsilon),
\]
and, neglecting $O(\epsilon)$ terms, we compute $L^2$ norms,
\[
\|\partial^\alpha u_1\|_{L^2}^2 = \|\partial^\alpha \tilde u\|_{L^2}^2
+ 2\epsilon  \int \partial^\alpha \tu \cdot \cA^j(\tu) \partial^\alpha \partial_j \tu \, dx 
+ \epsilon^2 \|  A^j(\tu) \partial^\alpha \partial_j \tu \|^2_{L^2} .
\]
The last $L^2$ norm has size $O(\epsilon)$ in view of property (a) above.
In the integral, on the other hand, we use the symmetry of $\cA$ to integrate by parts,
\[
2\int \partial^\alpha \tu \cdot \cA^j(\tu) \partial^\alpha \partial_j \tu \, dx =
- \int   \partial^\alpha \tu   \cdot \partial_j \cA^j(\tu) \partial^\alpha  \tu \, dx  ,
\]
which can again be estimated by $\lesssim_A B \| \tu\|_{H^N}^2 $. Thus we obtain 
\[
\| u_1 \|_{H^N}^2  \lesssim_A     (1+ \epsilon B) \| \tu\|_{H^N}^2,
\]
as desired, as $B$ can be estimated by the Sobolev norm of $u_0$ by Sobolev embeddings.

It remains to consider (ii)', where, by (c) above, it suffices 
to show that 
\[
\| \cA^j(u) \partial_j u - \cA^j(\tu) \partial_j \tu \|_{L^2}
\lesssim_M \epsilon .
\]
This is a soft argument, where we simply write
\[
\| \cA^j(u) \partial_j u - \cA^j(\tu) \partial_j \tu \|_{L^2}
\lesssim_M \| \cA(u) - \cA(\tu)\|_{L^2} + \| \partial_j u - \partial_j \tu \|_{L^2} \lesssim_M \| u - \tu\|_{H^1},
\]
where the $H^1$ norm on the right is bounded by interpolating (c) 
above with the uniform $H^N$ bound provided by (b). This requires 
$N \geq 2$.

\subsection{Other strategies}
Most of the other strategies to prove existence of solutions are based on 
constructing approximate flows, and solutions are obtained as limits of 
solutions to the approximate flows. There are two such methods which are more widely used.

\bigskip
\emph{a) Parabolic regularization.} Here one uses a parabolic regularization of the original flow \eqref{gen-eq}, defining the approximate solutions $u^\epsilon$ by
\[
u^\epsilon_t = N(u^\epsilon) - \epsilon (-\Delta)^k u^\epsilon, \qquad u(0) = u_0,
\]
where the correct choice for the parabolic term seems to be to double 
the order of the original equation.  These problems can often 
be solved for a short, $\epsilon$ dependent time, as semilinear problems, with a direct,
fixed point argument. However, in doing this, the main challenge 
is to prove uniform in $\epsilon$ bounds for these approximate flows.
This sometimes requires more careful choices of the regularization term, to make it fit better with the geometry of the problem. 

\bigskip

\emph{b) Galerkin approximation.} Here the idea is to work with a low frequency  projector in the equation, e.g. of the type
\[
u_t = P_{<h} N(P_{<h} u)
\]
with $h \to \infty$, see e.g. the example in \cite{Taylor}. The local solvability for this evolution becomes trivial as this evolution is an ordinary differential equation in a Hilbert space, but the challenge is again to prove uniform in $\epsilon$ bounds for these approximate flows. The double use of the projector above is a choice that usually facilitates achieving this objective.
Depending on the problem, this may require careful choices for the frequency projectors,
adapted to the problem.

\section{Rough solutions as limits of smooth solutions}
\label{s:rough}

Here we explore the idea of constructing rough solutions as limits of smooth solutions.
There are at least two good reasons to do this, which we discuss in order:

\begin{enumerate}
    \item In quasilinear problems one does not expect any sort of uniformly continuous dependence of solutions on the initial data, so the continuity of the flow map
    becomes a purely qualitative assertion. However, one can still ask for a quantitative way of comparing solutions, and such a quantitative venue is found by using 
    the regular approximations as a convenient proxy. This is discussed in the last section.
    
 \item It is also often the case that more regular solutions are sometimes easier to 
 produce, and in such situations, obtaining the rough solutions as limits of smooth solutions might be the only option.  This is particularly the case in problems where
 the state space is not a linear space, such as Schr\"odinger maps \cite{McGahagan}, Yang-Mills  or other problems with a nontrivial gauge structure.  See also \cite{no-rel} for an implementation of this idea in  a free boundary problem. This is because in such problems it is always easier to 
 obtain estimates for the linearized equations, or at least to compare exact solutions, rather than to cook  up a constructive scheme which is consistent with the geometry.
 \end{enumerate}

To make this analysis quantitative, it is very useful to track the flow of energy
between different frequencies. Whereas energy cascades (energy migration to higher frequencies) have long been associated with blow-up phenomena, well-posedness should correspond to a lack thereof. To quantify this, we will use Tao's notion of 
frequency envelopes.

\subsection{Frequency envelopes }

Frequency envelopes, introduced by Tao (see for example \cite{tao2000global}), are a very useful device in order 
to track the evolution of the energy of solutions between dyadic energy shells.
As there is always nearby leakage between the dyadic shells in nonlinear flows,
on needs to do this in a more stable way, rather than look directly at the exact amount of energy in every shell. 

This is realized via the following definition:

\begin{definition}
We say that $\{c_k\}_{k\geq 0} \in \ell^2$ is a frequency envelope for a function $u$ in $H^s$
if we have the following two properties:

a) Energy bound:
\begin{equation}
\|P_k u\|_{H^s} \leq c_k   , 
\end{equation}

b) Slowly varying
\begin{equation}\label{fe-delta}
\frac{c_k}{c_j} \lesssim 2^{\delta|j-k|} , \quad j,k\in \mathbb{N}.
\end{equation}
\end{definition}
Here $P_k$ represent the standard Littlewood-Paley projectors, and $\delta$ is a positive constant, which is taken small enough in order to account for energy leakage between nearby frequencies. 

One can also try to limit from above the size of a frequency envelope, for instance 
by requiring that 
\[
\| u\|_{H^s}^2 \approx \sum c_k^2.
\]
We call such envelopes \emph{sharp}. Such frequency envelopes always exist, for instance one can take 
\[
c_k = \sup_j 2^{-\delta|j-k|} c_j.
\]
For a better understanding see Figure \ref{fig:fe} below, where the actual dyadic norms, indicated by red bullets on a logarithmic scale, are lifted (based on the above formula) to a slowly varying frequency envelope, indicated by the green circles.
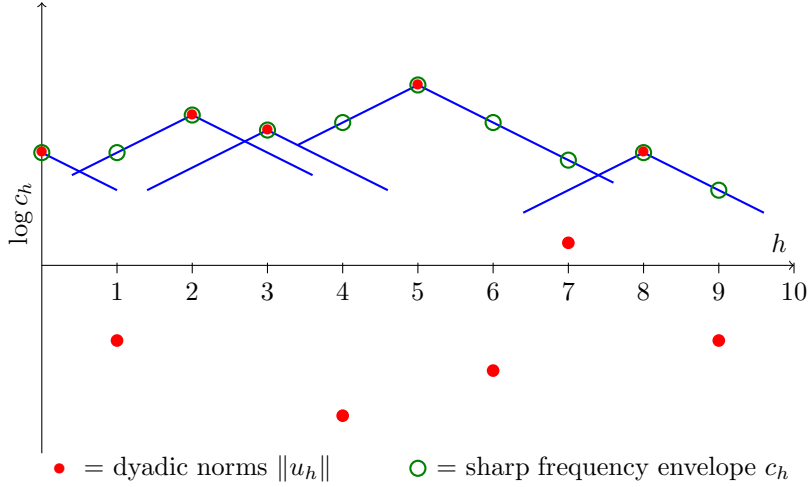
\begin{figure}
    \centering
  
    \begin{tikzpicture}
	\draw[->] (0,0) -- coordinate (x axis mid) (10,0);
    	\draw[->] (0,-2.5) -- coordinate (y axis top) (0,3.5);
    	\foreach \x in {1,...,10}
     		\draw (\x,1pt) -- (\x,-3pt)
			node[anchor=north] {\x};
    	

	
	\node at (9.8, .3) {$h$};
	\node[rotate=90, above=0.2cm] at (y axis top) {$\log c_h$};

	\draw[blue, thick]  (0,1.5)
    [sharp corners] -- (1,1);
    
    \draw[green1, thick] (0,1.5) circle [radius=0.1] node {{\color{red}$\bullet$}};

\draw[blue, thick] (0.4,1.2) -- (2,2)
    [sharp corners] -- (3.6,1.2);
    
    \draw[green1, thick] (2,2) circle [radius=0.1] node {{\color{red}$\bullet$}};

\draw[blue, thick] (1.4,1) -- (3,1.8)
    [sharp corners] -- (4.6,1);
    \draw[green1, thick] (3,1.8) circle [radius=0.1] node {{\color{red}$\bullet$}};
    
    \draw[blue, thick] (3.4,1.6) -- (5,2.4)
    [sharp corners] -- (7.6,1.1);
\draw[green1, thick] (5,2.4) circle [radius=0.1] node {{\color{red}$\bullet$}};

    \draw[blue, thick] (6.4,0.7) -- (8,1.5)
    [sharp corners] -- (9.6,.7);
\draw[green1, thick] (8,1.5) circle [radius=0.1] node {{\color{red}$\bullet$}};

\draw[red, fill, thick] (1,-1) circle [radius=0.07] node {};
\draw[green1, thick] (1,1.5) circle [radius=0.1];

\draw[red, fill, thick] (4,-2) circle [radius=0.07] node {};
\draw[green1, thick] (4,1.9) circle [radius=0.1] ;

\draw[red, fill, thick] (6,-1.4) circle [radius=0.07] node {};
\draw[green1, thick] (6,1.9) circle [radius=0.1] ;

\draw[red, fill, thick] (7,.3) circle [radius=0.07] node {};
\draw[green1, thick] (7,1.4) circle [radius=0.1] ;

\draw[red, fill, thick] (9,-1) circle [radius=0.07] node {};
\draw[green1, thick] (9,1) circle [radius=0.1] ;

\node at (2,-2.7) {{\color{red} $\bullet \ $} = dyadic norms $\Vert u_h\Vert$  };

\draw[green1, thick] (5,-2.7) circle [radius=0.1] node[anchor=west] { {\color{black} \ = sharp frequency envelope $c_h$}};
\end{tikzpicture}

\caption{Construction of sharp frequency envelopes. }

    \label{fig:fe}

\end{figure}

We will use frequency envelopes in order to track the evolution of energy in time as follows: we start with a sharp frequency envelope for the initial data, and then seek to show that we can propagate this frequency envelope to the solutions to our quasilinear flow, at least for a short time.

\begin{remark}
One alternative here is to unbalance the choice of $\delta$ in \eqref{fe-delta}, asking for 
a small $\delta$ if $k < j$, but replacing $\delta$ with a large constant for $k > j$.
This heuristically corresponds to a better control of leakage to higher frequencies,
and it is useful in order to deal with higher regularity properties also within the frequency envelope set-up.  
\end{remark}

\subsection{Regularized data}

Consider an initial data $u_0 \in H^s$ with size $M$, and let $\left\{c_k\right\}_{k\geq 0}$ be a sharp frequency envelope for $u_0$ in $H^s$.  For $u_0$ we consider a family of regularizations
$u_0^h \in H^\infty:=\cap _{s=0}^{\infty}H^s$ at frequencies $\lesssim 2^h$ where $h$ is a dyadic frequency parameter. This parameter 
can be taken either discrete or continuous, depending on whether 
we have access to difference bounds or only to the linearized equation.
Suppose we work with differences. Then the family $u^h_0$ can be taken 
to have similar properties to Littlewood-Paley truncations:

\begin{enumerate}[label=\roman*)]
\item Uniform bounds:
\begin{equation}
\| P_k u^h_0 \|_{H^s} \lesssim c_k.
\end{equation}

\item High frequency bounds:
\begin{equation}
\| u^h_0 \|_{H^{s+j}} \lesssim 2^{jh} c_h , \qquad j > 0.
\end{equation}

\item Difference bounds:
\begin{equation}
\| u^{h+1}_0- u^h_0 \|_{L^2} \lesssim 2^{-sh} c_h .
\end{equation}

\item Limit as $h \to \infty$:
\begin{equation}
u_0 = \lim_{h\to \infty} u_0^h  \qquad \text{ in } H^s.
\end{equation}
\end{enumerate}
Correspondingly, we obtain a family of smooth solutions $u^h$. 

Here in the simplest setting where the phase space is linear one may simply choose 
$u^h_0 = P_{< h} u_0$, which would have all the above properties. However, in geometric settings where the phase space is nonlinear, a more complex regularization method may be needed, for instance using a corresponding geometric heat flow, see \cite{tao-caloric}
or a variable scale regularization as in \cite{no-rel}.

\subsection{Uniform bounds }
Corresponding to the above family of regularized data, we obtain a family of smooth solutions $u^h$.  For this we can use the energy estimates as in Theorem~\ref{t:ee} to propagate Sobolev regularity for solutions as well as difference bounds as in Proposition~\ref{p:diff}. This yields a time interval $[0,T]$ where all these solutions
exist, and whose size $T$ depends only on $M = \|u_0\|_{H^s}$, where we have the following properties:

\begin{enumerate}[label=\roman*)]

\item High frequency bounds:
\begin{equation}\label{hf-bd}
\| u^h \|_{C(0,T;H^{s+j})} \lesssim 2^{jh} c_h , \qquad j > 0.
\end{equation}

\item Difference bounds:
\begin{equation}\label{diff-bd}
\| u^{h+1}- u^h\|_{C(0,T;L^2)} \lesssim 2^{-sh} c_h .
\end{equation}
\end{enumerate}

From \eqref{hf-bd} one may obtain a similar bound for the difference $u^{h+1}-u^h$.
Interpolating this with \eqref{diff-bd}, we also have 
\begin{equation}\label{interp-bd}
\| u^{h+1}- u^h\|_{C(0,T;H^m)} \lesssim  2^{-(s-m)h} c_h, \qquad m \geq 0. 
\end{equation}

One may use these bounds to establish uniform frequency envelope bounds
for $u^h$,
\begin{equation} \label{fe-bd}
\| P_k u^h \|_{C(0,T;H^{s})} \lesssim c_k 2^{-N(k-h)_+},
\end{equation}
on the same time interval which depends only on the initial data $H^s$ size.
This is a direct consequence of \eqref{hf-bd} for $k \geq h$, while if $k < h$
we can use the telescopic expansion
\[
u^h = u^k + \sum_{l = k}^{h-1}\left(  u^{l+1} - u^l\right),
\]
and use \eqref{hf-bd} for the first term and \eqref{diff-bd} for the differences.

\subsection{ The limiting solution} Consider now the convergence of $u^h$ as $h \to \infty$.
From the difference bounds \eqref{diff-bd} 
we obtain convergence in $L^2$ to a limit 
$u \in  C(0,T; L^2)$, with 
\[
\|u - u^h \|_{C(0,T;L^2)} \lesssim 2^{-sh}.
\]

On the other hand, expanding the difference as a telescopic sum we get
\[
u - u^h = \sum_{m = h}^\infty  u^{m+1} - u^m,
\]
where, in view of the above bounds \eqref{hf-bd} and \eqref{diff-bd}, 
each summand is essentially concentrated at frequency $2^m$, with $H^s$ size $c_m$ and exponentially decreasing tails.  This leads to 
\begin{equation}\label{rough-limit}
\| u - u^h \|_{C(0,T;H^s)} \lesssim c_{\geq h}:= \left(\sum_{m \geq h} c_m^2\right)^\frac12,
\end{equation}
so we also have convergence in $C(0,T;H^s)$. 

This type of argument plays multiple roles:
\begin{enumerate}
    \item It produces rough solutions as smooth solutions, justifying the earlier 
    assertion that it often suffices to carry out the initial construction of solutions only in a smooth setting.
    
    \item It establishes the continuity of solutions as $H^s$ valued flows, which is sometimes missing from the constructive proof of existence.
    
    \item It provides the quantitative bound \eqref{rough-limit} for the difference 
    between the rough and the smooth solutions, which plays a key role in the continuous dependence proof in the next section.
\end{enumerate}

\section{Continuous dependence }

Here we use frequency envelopes in order to prove continuous dependence
of the solution $u \in C(0,T;H^s)$ as a function of the initial data $u_0 \in H^s$,
and also discuss some historical alternatives.

\subsection{The continuous dependence proof}
Consider a sequence of initial data 
\[
u_{0j} \to u_0 \qquad in \ H^s, \quad s > \frac{d}{2}+1,
\]
and the corresponding solutions $u_j$, $u$ which exist with a uniform lifespan $[0,T]$,
where $T$ depends only on the initial data size $\|u_0\|_{H^s}$.
We will prove that $u_j \to u$ in $C(0,T;H^s)$. Once we have this property, it automatically extends to any larger time interval $[0,T_1]$ where the solution 
$u$ is defined and satisfies $u \in C(0,T_1;H^s)$.  This should be understood in the sense that for all large enough $j$, the solutions $u_j$ are defined in $[0,T_1]$,
with similar regularity, and the convergence holds as $j \to \infty$.

The difference bounds in Proposition~\ref{p:diff}  guarantee that $u_j \to u $ in $C(0,T;L^2)$. Since $u_j$ are uniformly bounded in $C(0,T;H^s)$, this also implies convergence in  $C(0,T;H^\sigma)$ for every $0 \leq \sigma < s$, but not for $\sigma = s$. 

It remains to consider the convergence in the strong topology, i.e. in $H^s$. 
Rather than trying to compare 
the solutions $u_j$ and $u$ directly, we will use as a proxy the approximate 
solutions $u_j^h$, respectively $u^h$. For these, we will take advantage 
of the fact that their initial data converge in all Sobolev norms,
\[
u_{0j}^h \to u_0^h  \qquad \text{in} \ H^\sigma, \quad 0 \leq \sigma < \infty.
\]
Hence, according to the preceding discussion, we have convergence of the regular solutions in all Sobolev norms, 
\[
u_{j}^h \to u^h  \qquad \text{in} \ C(0,T;H^\sigma), \quad 0 \leq \sigma < \infty.
\]
To compare the solutions $u$ and $u^j$ themselves, we use the triangle inequality,
\begin{equation}\label{diff-bd-j}
\| u_j - u \|_{C(0,T;H^s)} \lesssim \| u_j^h - u^h \|_{C(0,T;H^s)} +\| u^h - u \|_{C(0,T;H^s)}
+\| u_j^h - u_j \|_{C(0,T;H^s)}.
\end{equation}
The first term goes to zero as $j \to \infty$ for fixed $h$, while the second 
goes to zero as $h \to \infty$, but does not depend on $j$. It is the third term 
which is the problem, and for which we need to gain some smallness uniformly in $j$.

However, in the previous section we have learned to estimate such differences using frequency envelopes. Precisely, let $\left\{c_k\right\}_{k\geq 0}$, respectively $\left\{ c_k^j\right\}_{k\geq 0}$ be frequency envelopes
for the initial data $u_0$, respectively $u_0^j$ in $H^s$. Then, as we saw in the previous section, we can estimate the last two terms above in terms of frequency envelopes and obtain
\begin{equation}
\| u_j - u \|_{C(0,T;H^s)} \lesssim \| u_j^h - u^h \|_{C(0,T;H^s)} + c_{\geq h}
+ c^j_{\geq h}.
\end{equation}
The important observation is that the convergence $u_{0j} \to u_0$ in $H^s$ allows us to choose the frequency envelopes  $c$, respectively $c^j$ so that 
\[
c^j \to c \qquad \text{in } \ell^2.
\]
This implies that 
\[
\lim_{j \to \infty} c^j_{\geq h} = c_{\geq h}.
\]
Hence, passing to the limit $j \to \infty$ in the relation \eqref{diff-bd-j}, we obtain 
\begin{equation}
\limsup_{j \to \infty} \| u_j - u \|_{C(0,T;H^s)} 
\lesssim  c_{\geq h},
\end{equation}
and finally letting $h \to \infty$ we obtain 
\[
\lim_{j \to \infty} \| u_j - u \|_{C(0,T;H^s)} = 0,
\]
as desired.

\subsection{ Comparison with Kato and Bona-Smith}\
The more classical approach for continuous dependence goes back to Kato~\cite{Kato} as well as a variation due to Bona-Smith~\cite{BS}. We will briefly describe this approach using our notations and set-up; we caution the reader that the original arguments in these papers are not self-contained
and are instead mixed with the other parts of well-posedness proofs, so 
it is not exactly easy to correlate the papers with the description below.
In effect our discussion below is more closely based on the interpretations of Kato's 
work provided by Chemin~\cite{Chemin} and, even closer, by Tao~\cite{Tao}.

This also relies on the use of some sort of approximate solutions $u^h$. 
However, in this approach one aims to directly estimate 
the difference $u^h - u$ in $H^s$ in terms of the corresponding initial data.
 One might at first hope to directly track the 
difference $\| u^h - u\|_{C(0,T;H^s)}$, but this cannot work without knowledge 
that the low frequencies of the difference (i.e. below $2^h$) are better controlled.
So the better object to track turns out to be a norm of the form
\begin{equation}\label{control-norm-reg}
\| u^h - u\|_{H^s} + 2^{kh} \|  u^h - u\|_{H^{s-k}},  \end{equation}
where we recall that $k$ is the order of our nolinearity.
Here the second part can be estimated directly for any two $H^s$ solutions, see
Remark~\ref{r:sigma-diff}, so one can think of this as decoupled as a two step process. To better understand why this works, it is useful to write the equation 
for the difference $w= u^h-u$ in a paradifferential form
\begin{equation}\label{diff-eqn}
\partial_t w + T_{DN(u)} w = [F(u) - F(u^h)] + T_{DN(u)- DN(u^h)} u^h,
\end{equation}
which should essentially be thought of as a perturbation of the linear paradifferential flow, which can be estimated in all Sobolev spaces.
The $F$ difference is tame because $F$ admits Lipschitz bounds in all Sobolev spaces, so the issue is the last term. 

There there is seemingly a loss of $k$ derivatives, but these derivatives 
are applied to $u^h$, which has higher regularity bounds, so they yield losses of at most a $2^{kh}$ factor. But this factor can be absorbed by  the lower frequency  paradifferential coefficients given by $DN(u)- DN(u^h)$, in view of the $2^{kh}$ factor in \eqref{control-norm-reg}.
Here it is important that we wrote the equation using 
$T_{DN(u)}$ rather $T_{DN(u^h)}$ on the left, which allows us to use $u^h$ as the argument in the last term on the right.

In Kato's argument the same principle is used to get $H^s$ bounds not only 
for the difference $u^h - u$ but also for $u^h-v$ for an arbitrary solution $v$.
In Bona-Smith's, version, on the other hand, one estimates only $u^h - u$, but
the proof is more roundabout in that $u^h$ is not only assumed to have regularized data, but also to solve a regularized equation, thus combining the existence and the continuous dependence arguments.

In our opinion, working with frequency envelopes has definite advantages:
\begin{itemize}
    \item It provides more accurate information on the solutions.
    \item It does not require any direct difference bounds in the strong $H^s$
    topology.
    \item By working with a continuous, rather than a discrete family of regularizations
    one can fully replace difference estimates by bounds for the linearized equation, 
    which is to be preferred in many cases, in particular in geometric contexts
    where the state space is an infinite dimensional manifold.
\end{itemize}

 \bibliographystyle{abbrv}

\end{document}